\documentclass[a4paper,10pt]{amsart}
\usepackage[dvipdfm]{hyperref}
\usepackage{mathrsfs,rawfonts}
\usepackage{latexsym,amssymb,amsmath,mathrsfs,amsbsy, amsthm}

 \newtheorem{thm}{Theorem}[section]
 \newtheorem{cor}[thm]{Corollary}
 \newtheorem{lem}[thm]{Lemma}

 \newtheorem{rem}[thm]{Remark}
 \numberwithin{equation}{section}

 \newcommand{\ep}{\epsilon}
 \newcommand{\de}{\delta}
 \newcommand{\nx}{\nabla_x}

 \newcommand{\Real}{\mathbb{R}}

 \newcommand{\norm}[1]{\Vert#1\Vert}
 \newcommand{\Norm}[1]{\left\Vert#1\right\Vert}

 \newcommand{\dt}{\frac d {dt}}
 
 \def\lg{\langle} \def\rg{\rangle}
\begin{document}

\title{Generalized Landau-Lifshitz Equation into $S^n$}

\author{Chong Song}
\email{songchong@amss.ac.cn}
\address{LMAM, School of Mathematical Sciences,Peking University,
Beijing 100871, P.R. China.}

\author{Jie Yu}
\email{yujie@amss.ac.cn}
\address{Academy of Mathematics and System Sciences,Chinese Academy
of Science, Beijing 100190, P.R. China.}
\thanks{Partially supported
by 973 project of China, Grant No. 2006CB805902.}

\maketitle

\begin{abstract}
In this paper,  a type of integrable evolution equation--the
generalized Landau-Lifshitz equation into $S^n$ is considered. We
deal with this equation from a geometric point of view by rewriting
it in a geometric form. Through the geometric energy method, we show
the global well-posedness of the corresponding Cauchy problem.
\end{abstract}
\section{Introduction}
\label{Sec:intro}

Let $u$ be a map from $S^1\times \Real$ to the $n$-dimensional
sphere $S^n$, which is embedded into $\Real^{n+1}$. In this paper
the following equation:
\begin{equation}\label{e0}
u_t = (u_{xx} + \frac 3 2|u_x|^2u)_x + \frac 3 2(u, Au) u_x, \quad
|u|^2 =1
\end{equation}
is considered. Here $A$ is a constant symmetric matrix and $(\cdot,
\cdot)$ denotes the standard inner product on $\Real^{n+1}$. Without
loss of generality we may assume that $A$ is a diagonal matrix, i.e.
$A=diag(r_{1},\cdots,r_{n})$. We will also use $\lg\cdot, \cdot\rg$
to denote the metric on $S^n$. It is compatible with the inner
product on $\Real^{n+1}$ in the sense that $\lg X,Y \rg = (X,Y)$ for
any tangent vector fields $X,Y\in TS^n$.

For $n = 1$, with the trigonometric parameterizations of a circle,
equation~(\ref{e0}) becomes a well-known model in the theory of
exactly integrable systems \cite{DC}. For $n = 2$,
equation~(\ref{e0}) defines an infinitesimal symmetry for the
well-known Noemann system \cite{V} describing the dynamics of a
particle on the sphere $S^{n}$ under the influence of field with the
quadratic potential $U=\frac{1}{2}(u,Au)$. Besides,
equation~(\ref{e0}) coincides with the higher symmetry of third
order for the famous Landau-Lifshitz equation
\begin{equation}\label{e00}
    u_{t}=u\times u_{xx}+(Au)\times u,\quad |u|^2=1.
\end{equation}
Here the symbol `$\times$' denotes a vector product.  Thus
system~(\ref{e0}) is a generalized Landau-Lifshitz equation into
$n$-dimensional sphere $S^n$.

In \cite{GS}, I. Z. Golubchik and V. V. Sokolov showed that for any
dimension $n$ and matrix $A$ the system~(\ref{e0}) is exactly
integrable by the inverse scattering method, and has an infinite
number of symmetries. After that, Meshkov and Sokolov \cite{MS} used
the symmetry approach to give a complete classification of
integrable vector evolution equations similar to
equation~(\ref{e0}). In 2008, S. Igonin, J. Van De Leur, G. Manno,
and V. TrushkovIn~\cite{ILMT} successfully applied the
Wahlquist-Estabrook method to recover the infinite-dimensional Lie
algebra related to this system. For more reference on this topic,
see also~\cite{A,BY,S}.

However, from the perspective of partial differential equations, one
may naturally propose the problem of well-posedness of the
corresponding Cauchy problem of equation~(\ref{e0}) in appropriate
Sobolev spaces. In this paper, we intend to provide such results.
Let $\nabla_{x}$ denote the covariant derivative
$\nabla_{\frac{\partial}{\partial x}}$ on the pull-back bundle
$u^{*}TS^{n}$ induced from the Levi-Civita connection $\nabla$ on
$S^{n}$, then
$$\nabla_{x}^{2}u_{x}=u_{xxx}+3(u_{x},u_{xx}) u+\langle
u_{x},u_{x}\rangle u_{x}.$$ Therefore, equation~(\ref{e0}) can be
rewritten as a geometric flow on $S^n$:
\begin{equation}\label{e2}
u_t = \nx^2 u_x + \frac 1 2 |u_x|^2 u_x + \frac 3 2 (u, Au) u_x.
\end{equation}
When $n=2$ and $A\equiv0$, equation (\ref{e2}) coincides with the
geometric KdV Flow defined by Sun and Wang~\cite{SW}. In this sense,
equation (\ref{e0}) is a type of generalization of the so-called KdV
Flow.

The Cauchy problem corresponding to equation~(\ref{e0}),
i.e.~(\ref{e2}) is
\begin{equation}\label{e:cauchy}
\left\{
\begin{aligned}
&u_t = \nx^2 u_x + \frac 1 2 |u_x|^2 u_x + \frac 3 2 (u, Au) u_x,\\
&u(0) = u_0,
\end{aligned}
\right.
\end{equation}
where $u_0$ is an initial map from $S^1$ into $S^n$. Now it is
natural for us to treat this problem as a geometric evolution
equation and implement the so-called geometric energy method (see
Section 2). This method relies heavily on the geometric structure of
equation~(\ref{e2}) and seems more intrinsic. With this powerful
tool, we prove the global well-posedness of Cauchy
problem~(\ref{e:cauchy}). Our main result is the following theorem.
\begin{thm}\label{t1}
Suppose $u_0 \in W^{k,2}(S^1,S^n)$ for $k\ge 3$, then the Cauchy
problem~(\ref{e:cauchy}) admits a unique global solution $u \in
L^\infty(\Real^+, W^{k,2}(S^1,S^n))$.
\end{thm}
We sketch our strategy as follows:

First, we prove the local existence of solution to Cauchy
problem~(\ref{e:cauchy}) by perturbing the system with a 4th order
term $-\ep\nx^3 u_x$, where $\ep>0$ is a small positive number.
Namely, we consider the following perturbed system:
\begin{equation}\label{e:perturb}
\left\{ \begin{aligned} &u_t = -\ep\nx^3 u_x + \nx^2 u_x +
\frac12|u_x|^2u_x + \frac32(u, Au)u_x,\\
&u(0) = u_0.
\end{aligned}
\right.
\end{equation}
This is a 4th order parabolic system and it is well-known that there
exists a unique local solution $u_\ep$ of (\ref{e:perturb}) with
smooth initial data. Then we use this solution to approximate the
desired solution of~(\ref{e:cauchy}) by vanishing the perturbing
term, i.e. by letting $\ep$ go to $0$. The key step is to establish
an uniform estimate of the solution $u_\ep$ for $\ep>0$, see
Lemma~\ref{l2}. With this estimate, we are able to show that $u_\ep$
converges to a limit map $u$ which is a local solution to Cauchy
problem~(\ref{e:cauchy}). Furthermore, a careful calculation yields
the uniqueness of the solution.

Next, instead of computing conservation laws of the integrable
system, we define the following `energy' integrals:
\begin{align*}
    E_{2}(u)
    &=\int_{S^{1}}|\nabla_{x}u_{x}|^{2}
    -\frac{1}{4}\int_{S^{1}}|u_{x}|^4
    -\frac{9}{4}\int_{S^{1}}(u,Au)|u_{x}|^2
    +\int_{S^{1}}(u_{x},Au_{x}),\\
    E_{3}(u)
    &=\int_{S^{1}}|\nabla_{x}^{2} u_{x}|^{2}
    -\int_{S^{1}}\langle u_{x},\nabla_{x}u_{x}\rangle^{2}
    -\frac{3}{2}\int_{S^{1}}|u_{x}|^{2}|\nabla_{x}u_{x}|^{2}.
\end{align*}
Then we show that these geometric energies satisfy semi-conversation
laws under the flow~(\ref{e2}), i.e.
\begin{align*}
\frac {d}{dt}E_2(u) &\le C(E_2(u)+1), \\
\frac {d}{dt}E_3(u) &\le C(E_3(u)+1),
\end{align*}
where $C$ is a constant depending only on the $W^{3,2}$-norm of the
initial data $u_0$, the matrix $A$ and the existing time of the
solution. Applying Gronwall's inequality, we get the bound of these
energies, which implies a global bound of the $W^{3,2}$-norm of the
solution. A standard argument then yields the global existence of
solution to Cauchy problem~(\ref{e:cauchy}).

Though we only discuss the situation when $u$ is a map from $S^1$ to
$S^n$, the same result also holds for $u$ mapping from the real line
$\Real^1$. Actually, one may check this by following the argument in
\cite{DW}. The crucial fact is that the interpolation inequality in
Theorem~\ref{t:interpolation} is scaling invariant, and hence the
main estimate in Lemma~\ref{l2} does not depend on the diameter of
the domain. Thus we have the following

\begin{thm}\label{t2}
Suppose $u_0 \in W^{k,2}(\Real^1,S^n)$ for $k\ge 3$, then the Cauchy
problem~(\ref{e:cauchy}) admits a unique global solution $u \in
L^\infty(\Real^+, W^{k,2}(\Real^1,S^n))$.
\end{thm}

The rest of this paper is arranged as follows: we first recall the
geometric energy method in Section~\ref{Sec:1}. Then we apply this
method to show the local existence and uniqueness of solution to the
Cauchy problem~(\ref{e:cauchy}) in Section~\ref{Sec:2} and
Section~\ref{Sec:3} respectively. Next we compute the
semi-conservation laws we need in Section~\ref{Sec:4}. At last, we
finish the proof of global existence in Section~\ref{Sec:5}.

\section{Geometric energy method}
\label{Sec:1}

In this section, we recall the geometric energy method. This method
was first introduced by Ding and Wang in their seminar paper
\cite{DW} to show the local well-posedness of the Schr\"odinger
flow. Then similar methods were employed to treat different kinds of
problems in geometric analysis. It is especially powerful when
applied to non-linear evolution equations. For example, Kenig,
etc.~\cite{K} showed the same method works efficiently for a
difference scheme to approach the Schr\"odinger flow equation. Song
and Wang~\cite{SoW} used a similar method to prove the local
well-posedness of the wave map with potential, which implies the
existence of Schr\"odinger solitons on Lorentzian manifolds.
Moser~\cite{M} also defined a geometric energy to deal with the
biharmonic map. Recently, Sun and Wang~\cite{SW} applied this method
to investigate the geometric KdV flow.

The geometric energy method starts with a kind of geometric
Sobolev-type norms defined on Riemannian vector bundles. Its main
idea is to derive a priori estimates of the geometric energies, i.e.
the geometric Sobolev norms of the solution. Since the geometric
norm naturally involves with the geometry of the underlying
manifolds, it seems more intrinsic to investigate these norms
instead of the classical Sobolev norms when dealing with specific
equations with geometric backgrounds. However, these norms are
non-linear in general and harder to be dealt with than the normal
ones, because the Sobolev embedding theorems fail to hold.
Fortunately, Ding and Wang~\cite{DW} discovered a generalized
Gagliardo-Nirenberg inequality for the geometric Sobolev norms which
plays the key role in the geometric energy method. Moreover, they
found these geometric norms are in some sense equivalent to the
normal Sobolev norms. We summarize their results with two theorems
in the rest part of this section.

Let $\pi:E \to M$ be a Riemannian vector bundle over an
$m$-dimensional closed Riemannian manifold $M$ and $D$ denote the
covariant derivative on $E$ induced by the Riemannian metric. Then
we can define a Sobolev norm which we denote by $H^{k,p}$ for any
$k\ge 1$ and $p>0$ via the bundle metric for every section $s\in
\Gamma(E)$ by
\[ \norm{s}_{H^{k,q}} = \sum_{l=0}^k \norm{D^l s}_{L^q}. \]

\begin{thm}\label{t:interpolation} (\cite{DW})
Suppose $s \in C^\infty(E)$ is a section where $E$ is a vector
bundle on $M$. Then we have
\begin{equation*}
\Norm{D^j s}_{L^p} \leq C\Norm{s}^a_{H^{k,q}}\Norm{s}^{1-a}_{L^r},
\end{equation*}
where $1\leq p,q,r\leq \infty$, and $j/k \leq a \leq 1 (j/k \leq a <
1$ if $q=m/(k-j) \neq 1)$ are numbers such that
\[ \frac 1 p = \frac j m + \frac 1 r + a(\frac 1 q - \frac 1 r - \frac k m). \]
The constant $C$ only depends on $M$ and the numbers $j,k,q,r,a$.
\end{thm}

\begin{cor}\label{c1}
Suppose $s \in C^\infty(E)$ is a section where $E$ is a vector
bundle on $S^1$. Then we have
\begin{equation}\label{ee6}
\Norm{s}_{L^\infty} \leq
C\Norm{s}^{\frac12}_{H^{1,2}}\Norm{s}^{\frac12}_{L^2}
\end{equation}
\end{cor}
\begin{proof}
Since $m=1$, just let $j=0, p=\infty, a=1/2, k=1, q=2, r=2$ and
apply Theorem~\ref{t:interpolation}.
\end{proof}

Especially, for a map $u\in C^\infty(S^1, N)$, the pull-back bundle
$u^*(TN)$ is a Riemannian vector bundle on 1-dimensional manifold
$M= S^1$. So the above inequality~(\ref{ee6}) applies for section
$s=\nx^l u_x\in \Gamma(u^*(TN))$ with $l\ge 0$, which yields
\begin{equation}\label{equ:int}
\Norm{\nx^l u_x}_{L^\infty} \le C\Norm{u_x}_{H^{l+1,2}}^{\frac
12}\Norm{u_x}_{L^{2}}^{\frac 12} \le C\Norm{u_x}_{H^{l+1,2}}.
\end{equation}

For any map $u$ from a $m$-dimensional Riemannian manifold $M$ to a
compact Riemannian manifold $N$ which can be embedded into a
Euclidean space $\Real^K$, we have two kinds of Sobolev norms--
namely, the above $H^{k,p}$ norms of section $Du = \nabla u \in
\Gamma(u^*TN)$ and the normal $W^{k, p}$ Sobolev norms of function
$u:M\to \Real^K$, i.e.
\[ \norm{u}_{W^{k,p}} = \sum_{l=0}^k \norm{\nabla^l u}_{L^q}, \]
where $\nabla$ denotes the covariant derivative of functions on $M$.
Ding and Wang showed that for $k> m/2$, the $H^{k,p}$ norm of $Du$
is equivalent to the $W^{k+1, p}$ norm of $u$. Precisely, we have

\begin{thm}\label{t:equivalence} (\cite{DW})
Assume that $k > m/2$. Then there exists a constant $C = C(N,k)$
such that for all $u \in C^\infty(M,N)$,
\[ \Norm{\nabla u}_{W^{k-1,2}} \leq C\sum_{i=1}^k \Norm{Du}^i_{H^{k-1,2}} \]
and
\[ \Norm{Du}_{H^{k-1,2}} \leq C\sum_{i=1}^k \Norm{\nabla u}^i_{W^{k-1,2}} \]
\end{thm}

In our case, $m=1$ and $k$ can be any positive integer. That means,
for all $k\ge 1$, the Sobolev norms $\norm{u_x}_{W^{k,2}}$ are
equivalent to the nonlinear norms $\norm{u_x}_{H^{k,2}}$ of the same
order.

\section{Local existence}
\label{Sec:2}

A basic property of the flow~(\ref{e2}) is that it preserves the
perturbed energy
\begin{equation}\label{eq:3.1}
    E_1(u) =
    \frac 1 2\int_{S^{1}} |u_{x}|^2dx + \frac 1 2\int_{S^{1}} ( u, Au )dx.
\end{equation}

\begin{lem}\label{l1}
$E_1$ is conserved under the flow~(\ref{e2}).
\end{lem}
\begin{proof}
A direct computation yields:
\begin{eqnarray*}
    \begin{split}
        &\frac 1 2\dt\int_{S^{1}}|u_x|^2\\
        = &
        \int_{S^{1}}\lg \nabla_{t}u_{x},u_{x}\rg
        =
        \int_{S^{1}}\lg \nx u_{t},u_{x}\rg\\
        = &
        -\int_{S^{1}}\lg \nx^{2}u_{x}+\frac{1}{2}|u_x|^2 u_{x}
        +\frac{3}{2}( u,Au) u_{x},\nx u_{x}\rg\\
        = &
        -\int_{S^{1}}\lg \nx^{2}u_{x},\nx u_{x}\rg
        -\frac{1}{2}\int_{S^{1}}|u_x|^2\lg
        u_{x},\nx u_{x}\rg - \frac{3}{2}\int_{S^{1}}(u,Au)\lg u_{x},\nx u_{x}\rg.\\
    \end{split}
\end{eqnarray*}
The first two terms vanishes, when integrating by parts. Thus,
\begin{equation}\label{e3}
\frac 1 2\frac{d}{dt}\int_{S^{1}}|u_{x}|^2 =
\frac{3}{2}\int_{S^{1}}|u_x|^2(u,Au_x).
\end{equation}
On the other hand, using equation~(\ref{e0}), we have
\begin{eqnarray}\label{e4}
    \begin{split}
        \frac 12\dt\int_{S^{1}} ( u, Au )
        = &
        \int_{S^{1}}( u_t,Au)\\
        = &
        -\int_{S^{1}} ( u_{xx} + \frac 3 2|u_x|^2u, Au_x) + \frac 3 2\int_{S^{1}}( u, Au)
        ( u_x, Au )\\
        = &
        -\frac 3 2\int_{S^{1}} |u_x|^2( u, Au_x).
    \end{split}
\end{eqnarray}
From (\ref{e3}) and (\ref{e4}), we get
\begin{equation*}
\dt E_1(u(t)) =0.
\end{equation*}
Lemma~\ref{l1} follows.
\end{proof}

Since the matrix $A$ is constant and $|u|=1$ on the sphere $S^n$, we
have
\begin{equation}\label{ee1}
\int_{S^{1}} ( u, Au )dx\le C.
\end{equation}
Thus we get the following easy corollary:

\begin{lem}\label{l3}
$\norm{u_x}_{L^2}$ is bounded under flow~(\ref{e2}). Moreover, the
following inequality holds:
\begin{equation}
\norm{u_x}_{L^\infty}\le C\norm{u_x}_{H^{1,2}}^{\frac 12}.
\end{equation}
\end{lem}
\begin{proof}
The first statement follows directly form Lemma~\ref{l1} and
(\ref{ee1}). Combining this and the interpolation
inequality~(\ref{ee6}), we get
\[ \norm{u_x}_{L^\infty}\le C\norm{u_x}_{H^{1,2}}^{\frac 12}\norm{u_x}_{L^2}^{\frac 12}
\le C\norm{u_x}_{H^{1,2}}^{\frac 12}. \]
\end{proof}

To attain the local existence, we approximate equation~(\ref{e2}) by
a 4th-order parabolic system:
\begin{equation}\label{e:approximate}
\left\{ \begin{aligned}   &u_t = -\ep\nx^3 u_x + \nx^2 u_x +
\frac12|u_x|^2u_x + \frac32(u, Au)u_x,\\
&u(0) = u_0.
\end{aligned}
\right.
\end{equation}
where $\ep>0$ is a small number.

Since $N= S^n$ is a submanifold of $\mathbb{R}^{n+1}$, $u$ could be
considered as a mapping from $S^1$ into $\mathbb{R}^{n+1}$. The
equation (\ref{e:approximate}) then becomes a fourth order parabolic
equation in $\mathbb{R}^{n+1}$ which is analogous to the heat flow
of biharmonic map. By the standard parabolic theory (See \cite{T},
for example), equation~(\ref{e:approximate}) admits a local solution
$u_\ep\in C^\infty([0, T_\ep)\times S^1, S^n)$, if the initial map
$u_0$ is smooth. Moreover, one can verify that $u(t)$ lies on the
sphere $S^n$ for any $t\in [0,T_\ep)$, if the initial map
does~\cite{W}.

Similarly to Lemma~\ref{l1}, one can prove the following lemma
through direct computation.
\begin{lem}\label{l11}
If $u_\ep:[0,T_\ep)\times S^1 \to S^n$ is a solution to the Cauchy
problem~(\ref{e:approximate}), then $$E_1(u_\ep(t)) \le E_1(u_0)),
\quad \forall t\in [0,T).$$
\end{lem}

Next we approximate the solution of the original
system~(\ref{e:cauchy}) by vanishing the perturbing term, i.e.
letting $\ep$ go to $0$. To achieve this goal, we need the following
lemma which provides an uniform estimate on the solutions of
(\ref{e:approximate}) for $\ep>0$ .

\begin{lem}\label{l2}
Let $u_0\in C^\infty(S^1,S^n)$ and $u\in C^\infty([0,T_\ep)\times
S^1,S^n)$ be a solution of (\ref{e:approximate}) with $\ep\in(0,1]$.
Then for any integer $k\ge 2$, there exists a $T_k>0$ which is
independent of $\ep$, such that
\begin{equation}\label{e:local-estimate}
\norm{\nx u(t)}_{H^{k,2}}\le C(k, \norm{\nx u_0}_{H^{k,2}}),\quad
t\in [0,T_k].
\end{equation}
\end{lem}
\begin{proof}
For any $l\ge 2$, we compute
\begin{eqnarray}\label{e14}
    \begin{split}
        &\frac 1 2\dt\int_{S^{1}} |\nx^l u_x|^2 = \int_{S^{1}} \lg \nabla_t\nx^l u_x,
        \nx^l u_x \rg\\
        = &
        \int_{S^{1}}\lg \nx^l \nabla_t u_x + Q_l(u_t, u_x), \nx^l u_x \rg\\
        = &
        \int_{S^{1}} \lg \nx^{l+1}(-\ep\nx^3 u_x + \nx^2 u_x), \nx^l u_x
        \rg\\
        &+\frac12\int_{S^{1}} \lg \nx^{l+1}(|u_x|^2u_x + 3(u, Au)u_x), \nx^l u_x
        \rg+\int_{S^{1}}\lg Q_l(u_t, u_x), \nx^l u_x \rg\\
        := &
        I_1 +I_2 +I_3,
    \end{split}
\end{eqnarray}
where $Q_l(u_t, u_x)$ denotes the curvature terms, and will be
treated later.

For the first term in (\ref{e14}), it is easy to see
\begin{equation}\label{e5}
I_1 = -\ep \int_{S^{1}} \lg \nx^{l+4}u_x, \nx^l u_x\rg +
\int_{S^{1}} \lg
        \nx^{l+3}u_x, \nx^l u_x\rg = -\ep\int_{S^{1}}|\nx^{l+2}u_x|^2.
\end{equation}
For the second term, we have
\begin{equation}\label{e8}
\begin{aligned}
I_2 &= \frac12\int_{S^{1}} \lg \nx^{l+1}(|u_x|^2u_x), \nx^l u_x \rg
+
\frac32\int_{S^{1}}\lg \nx^{l+1}((u, Au)u_x), \nx^l u_x \rg\\
&:= J_1 + J_2
\end{aligned}
\end{equation}
We first estimate $J_1$. After differentiating, we get
\begin{equation}\label{e6}
\begin{aligned}
J_1 \le &|\int_{S^{1}}\lg \nx^{l+1} u_x, u_x \rg \lg u_x, \nx^l u_x \rg|+ \frac 12|\int_{S^{1}}\lg u_x,u_x\rg \lg\nx^{l+1} u_x, \nx^l u_x\rg|\\
&+C\sum_{a,b,c} \int_{S^{1}}|\nx^a u_x||\nx^b u_x||\nx^c u_x||\nx^l
u_x|,
\end{aligned}
\end{equation}
where the sum is taken over all integers $a, b, c$ satisfying
\begin{equation}\label{e7}
a+b+c = l+1, \text{ and }~ l\ge a, b, c\ge 0.
\end{equation}
For the first two terms in (\ref{e6}), integrating by parts, we have
\begin{equation}\label{e12}
\int_{S^{1}}\lg \nx^{l+1} u_x, u_x \rg \lg u_x, \nx^l u_x \rg =
-\int_{S^{1}}\lg \nx^l u_x, \nx u_x \rg \lg u_x, \nx^l u_x \rg,
\end{equation}
and
\begin{equation*}
\int_{S^{1}}\lg u_x,u_x\rg \lg\nx^{l+1} u_x, \nx^l u_x\rg =
-\int_{S^{1}} \lg u_x, \nx u_x \rg \lg \nx^l u_x, \nx^l u_x\rg,
\end{equation*}
Thus these two terms are of the same form as the summation term in
(\ref{e6}), only with $a=l$ and $b=1, c=0$. Now we may recall
inequality~(\ref{equ:int}) to estimate
\begin{align*}
\int_{S^{1}}|\nx^l u_x||\nx u_x||u_x||\nx^l u_x| &\le \norm{\nx
u_x}_{L^\infty}\norm{u_x}_{L^\infty}\norm{\nx^l u_x}_{L^2}^2\\
&\le C\norm{u_x}_{H^{l,2}}^4.
\end{align*}
Note that we used the assumption $l\ge 2$ here. However for $l = 1$,
this term is also bounded by
\begin{equation}\label{e15}
\int_{S^{1}}|\nx u_x|^3|u_x| \le C\norm{u_x}_{H^{2,2}}^4.
\end{equation}
For the other terms of the summation where $a, b, c\le l-1$, we have
\begin{align*}
\int_{S^{1}}|\nx^a u_x||\nx^b u_x||\nx^c u_x||\nx^l u_x| &\le
\norm{\nx^a u_x}_{L^\infty}\norm{\nx^b u_x}_{L^\infty}\norm{\nx^c
u_x}_{L^\infty}\norm{\nx^l u_x}_{L^1}\\
&\le C\norm{u_x}_{H^{l,2}}^4.
\end{align*}
Hence, we find $J_1$ bounded by
\begin{equation}\label{e9}
J_1 \le C\norm{u_x}_{H^{l,2}}^4.
\end{equation}

Similarly, the second term $J_2$ satisfies
\begin{equation*}
\begin{aligned}
J_2 \le & 3|\int_{S^{1}}(D_x^{l+1} u, Au ) \lg u_x, \nx^l u_x \rg|+ \frac 32|\int_{S^{1}}( u, Au) \lg\nx^{l+1} u_x, \nx^l u_x\rg|\\
&+C\sum_{a,b,c} \int_{S^{1}}|D_x^a u||D_x^b u||\nx^c u_x||\nx^l
u_x|,
\end{aligned}
\end{equation*}
where $D_x$ denotes the derivative of functions, and the sum is
taken over all integers $a, b, c$ satisfying~(\ref{e7}). Now we may
treat $J_2$ in almost the same way as $J_1$, except that we shift to
the classical Sobolev inequalities for the terms in the Euclidean
inner product $(\cdot, \cdot)$ this time. Still, we can obtain
\begin{equation}\label{e10}
J_2 \le C\norm{u_x}_{W^{l,2}}^2\norm{u_x}_{H^{l,2}}^2 \le
C\norm{u_x}_{H^{l,2}}^4,
\end{equation}
since the $W^{l,2}$ and $H^{l,2}$ Sobolev norms are equivalent by
Theorem~\ref{t:equivalence}.

Combining (\ref{e8}),(\ref{e9}),(\ref{e10}), we get
\begin{equation}\label{e11}
I_2 \le C\norm{u_x}_{H^{l,2}}^4.
\end{equation}

Finally we turn to the third term, i.e. the curvature term
\begin{equation*}
I_3 = \int_{S^{1}}\lg Q_l(u_t, u_x), \nx^l u_x \rg
=\sum_{a,b,c}C_{a,b,c}\int_{S^{1}}\lg R(\nx^a u_t, \nx^b u_x)\nx^c
u_x, \nx^l u_x \rg,
\end{equation*}
where the sum is taken over all integers $a, b, c$ satisfying
\begin{equation*}
a+b+c = l-1, \text{ and }~ l-1\ge a, b, c\ge 0,
\end{equation*}
and $C_{a,b,c}$ are combination numbers bounded by a constant $C_l$
only depending on $l$. Substituting $u_t$ by
equation~(\ref{e:approximate}), we get
\begin{align*}
I_3 &\le C_l\sum_{a,b,c} \left\{-\ep\int_{S^{1}}\lg R(\nx^{a+3} u_x, \nx^b u_x)\nx^c u_x, \nx^l u_x \rg \right.\\
    &\qquad+ \int_{S^{1}}\lg R(\nx^{a+2} u_x, \nx^b u_x)\nx^c u_x, \nx^l u_x \rg\\
    &\qquad\left.+ \frac12\int_{S^{1}}|\nx^a (|u_x|^2u_x + 3(u, Au)u_x)||\nx^b u_x||\nx^c u_x||\nx^l
    u_x|\right\}\\
    &:= C_l(J_4+J_5),
\end{align*}
where $J_4$ denotes the higher-order terms with $\nx^d u_x$, $d\ge
l+1$, while $J_5$ denotes the summation of the rest terms. We only
need to deal with $J_4$ here, since all the lower-order terms in
$J_5$ can be bounded. Namely, after a similar argument as the
estimate of $I_2$, which we omit here, we can obtain
\begin{equation*}
J_5 \le C\norm{u_x}_{H^{l,2}}^6.
\end{equation*}
On the other hand, there are four terms in $J_4$, i.e.
\begin{align*}
J_4 &= -\ep(\int_{S^{1}}\lg R(\nx^{l+2} u_x, u_x)u_x, \nx^l u_x \rg+ \int_{S^{1}}\lg R(\nx^{l+1} u_x, \nx u_x) u_x, \nx^l u_x \rg\\
    &\quad+ \int_{S^{1}}\lg R(\nx^{l+1} u_x, u_x) \nx u_x, \nx^l u_x \rg)+ \int_{S^{1}}\lg R(\nx^{l+1} u_x, u_x) u_x, \nx^l u_x
    \rg.
\end{align*}
The last term can be handled as we have done in (\ref{e12}). For the
first term, it follows from Young's inequality that for any
$0<\de<1$
\begin{align*}
|\ep\int_{S^{1}}\lg R(\nx^{l+2} u_x, u_x)u_x, \nx^l u_x \rg| &\le
\ep\int_{S^{1}}
|\nx^{l+2} u_x||\nx^l u_x||u_x|^2\\
&\le \ep \de \int_{S^{1}}|\nx^{l+2} u_x|^2 + \frac \ep\de
\int_{S^{1}}|\nx^l
u_x|^2|u_x|^4\\
&\le \ep \de \int_{S^{1}}|\nx^{l+2} u_x|^2 + \frac C\de
\norm{u_x}_{H^{l,2}}^6.
\end{align*}
Similarly, for the rest two terms in $J_4$,
\begin{align*}
&\ep(\int_{S^{1}}\lg R(\nx^{l+1} u_x, \nx u_x) u_x, \nx^l u_x \rg+
\int_{S^{1}}\lg
R(\nx^{l+1} u_x, u_x) \nx u_x, \nx^l u_x \rg)\\
\le &\ep \de \int_{S^{1}}|\nx^{l+1} u_x|^2 + \frac{4\ep}{\de}
\int_{S^{1}}|\nx^l u_x|^2|\nx
u_x|^2|u_x|^2\\
\le & \ep \de \int_{S^{1}}|\nx^{l+1} u_x|^2 + \frac C\de
\norm{u_x}_{H^{l,2}}^6.
\end{align*}
Note that we assume $l\ge 2$ again in the above. As for $l = 1$, it
is obviously bounded by $C\norm{u_x}_{H^{2,2}}^6$. Hence, if we
choose $\de = 1/C_l$, we get
\begin{equation*}
J_4 \le \frac \ep {C_l} \int_{S^{1}}|\nx^{l+2} u_x|^2 + \frac\ep
{C_l}\int_{S^{1}}|\nx^{l+1} u_x|^2 + C\norm{u_x}_{H^{l,2}}^6,
\end{equation*}
which now implies
\begin{equation}\label{e13}
I_3\le \frac\ep 2 \int_{S^{1}}|\nx^{l+2} u_x|^2 + \frac\ep 2
\int_{S^{1}}|\nx^{l+1} u_x|^2 + C\norm{u_x}_{H^{l,2}}^6.
\end{equation}
Correspondingly, for $l =1$ we have
\begin{equation}\label{e16}
I_3\le \frac\ep 2 \int_{S^{1}}|\nx^{3} u_x|^2 +
C\norm{u_x}_{H^{2,2}}^6.
\end{equation}

Combining~(\ref{e14}),(\ref{e5}),(\ref{e11}) and (\ref{e13}), we
arrive at
\begin{equation}\label{e17}
\frac12 \dt |\nx^l u_x|^2 \le - \frac\ep 2 \int_{S^{1}}|\nx^{l+2}
u_x|^2 + \frac\ep 2 \int_{S^{1}}|\nx^{l+1} u_x|^2 +
C\norm{u_x}_{H^{l,2}}^4 + C\norm{u_x}_{H^{l,2}}^6,
\end{equation}
,for any $l \ge 2$. Also for the case $l=1$, the above arguments
together with (\ref{e15}),(\ref{e16}) shows that
\begin{equation}\label{e18}
\frac 12 \dt \int_{S^{1}}|\nx u_x|^2 \le -\frac \ep 2
\int_{S^{1}}|\nx^3 u_x|^2 + C\norm{u_x}_{H^{2,2}}^4+
C\norm{u_x}_{H^{2,2}}^6.
\end{equation}
Summing inequality~(\ref{e17}) from $l = 2$ to $k$, and putting
(\ref{e18}) and Lemma~\ref{l11} into account, we conclude that there
exists a constant $C_k$ depending only on $k$ such that
\begin{equation*}
\dt \norm{u_x}_{H^{k,2}}^2 \le C_k(1+\norm{u_x}_{H^{k,2}}^2)^3
\end{equation*}
for any $k\ge 2$.

For this ordinary differential inequality with respect to
$\norm{u_x}_{H^{k,2}}^2$, whose initial data is $u(0) = u_0$, there
exist $T_k = T(\norm{\nx u_0}_{H^{k,2}})$ for all $k\ge 2$ such that
\[ \norm{u_x(t)}_{H^{k,2}}^2 \le C(k, \norm{\nx u_0}_{H^{k,2}}), \quad \forall t\in (0, T_k]. \]
Thus we complete the proof of the lemma.
\end{proof}

\begin{rem}
Actually, we can say more for $k > 2$. With the same procedure but
more careful treatment with the interpolations (See ~\cite{DW,SW}),
one can prove that
\begin{equation}\label{e19}
\dt \norm{u_x}_{H^{k,2}}^2 \le
C(k,\norm{u_x}_{H^{k-1,2}})(1+\norm{u_x}_{H^{k,2}}^2).
\end{equation}
One important fact is that the expression (\ref{e19}) is a linear
differential inequality for $\norm{u_x}_{H^{k,2}}^2$, therefore the
existing time of the solution only depends on the existing time of
$\norm{u_x}_{H^{k-1,2}}$, which in turn equals to $T_0 = T(\norm{\nx
u_0}_{H^{2,2}}) = T(\norm{u_0}_{W^{3,2}})$ by induction.
\end{rem}

Now we are ready to prove the local existence of the solution to
Cauchy problem~(\ref{e:cauchy}). We state this result in a separate
theorem.

\begin{thm}\label{t3}
Suppose $u_0 \in W^{k,2}(S^1,S^n)$ for $k\ge 3$, then the Cauchy
problem~(\ref{e:cauchy}) admits a local solution $u \in
L^\infty([0,T), W^{k,2}(S^1,S^n))$ for some positive number $T>0$.
Moreover, if the initial data $u_0$ is smooth, so is the solution
$u$.
\end{thm}
\begin{proof}
We first assume the initial map $u_0$ is smooth. By Lemma~\ref{l2},
we know that the Cauchy problem (\ref{e:approximate}) admits a
unique smooth solution $u_\ep \in C^\infty([0,T)\times S^1, S^n)$
which satisfies the estimates~(\ref{e:local-estimate}) with $T$ only
depending on $\norm{u_0}_{W^{3,2}}$. Then by
Theorem~\ref{t:equivalence}, we have for any integer $p>0$ and
$\ep\in(0,1]$:
\begin{eqnarray*}
\sup_{t\in[0,T]} ||u_\ep||_{W^{p,2}(N)}\le C_p(N,u_0),
\end{eqnarray*}
where $C_p(N,u_0)$ does not depend on $\ep$. Thus, by sending
$\ep\rightarrow 0$ and applying the embedding theorem of Sobolev
spaces to $u$, we have $u_\ep\rightarrow u\in C^p(S^1\times[0,T])$
for any $p$. It is easy to check that $u$ is a solution to the
Cauchy problem (\ref{e:cauchy}).

Next, if $u_0 \in W^{k,2}(S^1,S^n)$ for $k\ge 3$, then we can always
choose a sequence $u^i_0 \in C^\infty(S^1,S^n)$, such that $u^i_0$
converges to $u_0$ in $W^{k,2}(S^1,S^n)$. Now for each initial data
$u^i_0$, we have a solution $u^i\in C^\infty([0,T^i)\times S^1,
S^n)$ to Cauchy problem~(\ref{e:cauchy}). They all satisfy
estimate~(\ref{e:local-estimate}), i.e.
\begin{equation*}
\norm{\nx u^i(t)}_{H^{k-1,2}}\le C(k, \norm{\nx u_i}_{H^{k-1,2}})
\le C(k, \norm{\nx u_0}_{H^{k-1,2}}),\quad t\in [0,T^i])
\end{equation*}
Moreover, we have an uniform lower bound on the existing time $T^i$,
i.e.
\begin{equation*}
T^i = T(\norm{u^i_0}_{W^{3,2}}) \ge T=T(\norm{u_0}_{W^{3,2}}).
\end{equation*}
Again by Theorem~\ref{t:equivalence}, we have
\begin{equation*}
\sup_{t\in[0,T]} \norm{u^i(t)}_{W^{k,2}}\le
C(k,\norm{u_{0}}_{W^{k,2}}).
\end{equation*}
Therefore there exists a subsequence $\{u^j\}$ and a map $u\in
L^\infty([0,T],W^{k,2}(S^1,S^n))$ such that
\begin{eqnarray*}
u^j\rightarrow u\quad[\text{weakly}^*]\quad\text{in}\quad
L^\infty([0,T],W^{k,2}(S^1,S^n)).
\end{eqnarray*}

It is easy to verify that the limit map $u$ we get above is indeed a
strong solution of Cauchy problem~(\ref{e:cauchy}).
\end{proof}

\section{Uniqueness}
\label{Sec:3}

This section is devoted to prove the following uniqueness theorem.
This result relies heavily on the geometric structure of
equation~(\ref{e2}).

\begin{thm}\label{t4}
Suppose $u_0 \in W^{k,2}(S^1,S^n)$ for $k\ge 3$, then the solution
of Cauchy problem~(\ref{e:cauchy}) is unique.
\end{thm}
\begin{proof}
Suppose $u$ and $v$ are two solutions of Cauchy
problem~(\ref{e:cauchy}) with same initial data $u_0$. Let $w=u-v$,
then $w$ satisfies the following equation:
\begin{equation}\label{u1}
\begin{split}
w_t =&~ w_{xxx} + 3[(u_x,u_{xx})u - (v_x, v_{xx})v] \\
        &\quad+  \frac 32 [|u_x|^2 u_x - |v_x|^2 v_x ]+ \frac 32 [(u,Au)u_x - (v,Av)v_x]\\
    =&: w_{xxx} + 3I_1 + \frac 32 I_2 +\frac 32 I_3.
\end{split}
\end{equation}
By inserting intermediate terms, we have
\begin{equation*}
\begin{split}
I_1 =&~ (u_x,u_{xx})u  - (u_x, v_{xx})u + (u_x,v_{xx})u  - (v_x, v_{xx})u + (v_x,v_{xx})u  - (v_x, v_{xx})v \\
    =&~ (u_x, w_{xx})u + (w_x, v_{xx})u + (v_x, v_{xx})w.
\end{split}
\end{equation*}
Similarly, for the last two terms in~(\ref{u1}), we have
\begin{equation*}
I_2 =~ (w_x, u_x)u_x + (w_x, v_x)u_x + (v_x, v_x)w_x,
\end{equation*}
and
\begin{equation*}
I_3 =~ (w, Au)u_x + (w, Av)u_x + (v, Av)w_x.
\end{equation*}

Now, we are going to compute $\dt \norm{w}_{W^{1,2}}$. First, it's
easy to see
\begin{equation}\label{u5}
\begin{split}
\frac 12 \dt \norm{w}_{L^2}^2
    =&~\int_{S^1}(w_t, w) \\
    =&~\int_{S^1}(w_{xxx}, w) + 3\int_{S^1}(u_x, w_{xx})(u, w)\\
     &~+3\int_{S^1}(w_x, v_{xx})(u, w) + 3\int_{S^1}(v_x, v_{xx})|w|^2\\
     &~+\frac 32 \int_{S^1}(w_x, u_x)(u_x,w) + (w_x, v_x)(u_x,w) + (v_x, v_x)(w_x,w)\\
     &~ + \frac 32 \int_{S^1}(w, Au)(u_x,w) + (w, Av)(u_x,w)+(v, Av)(w_x,w)\\
    \le&~ C(\norm{u}_{W^{3,2}} +\norm{v}_{W^{3,2}})\norm{w}_{W^{1,2}}^2.
\end{split}
\end{equation}
Next, we claim that
\begin{equation}\label{u6}
\begin{split}
&~-\frac 12 \dt \norm{w_x}_{L^2}^2\\
    =&~\int_{S^1}(w_t, w_{xx}) \\
    =&~\int_{S^1}(w_{xxx}, w_{xx}) + 3\int_{S^1}(u_x, w_{xx})(u, w_{xx})\\
     &~+3\int_{S^1}(w_x, v_{xx})(u, w_{xx}) + 3\int_{S^1}(v_x, v_{xx})(w,w_{xx})\\
     &~+\frac 32 \int_{S^1}(w_x, u_x)(u_x,w_{xx}) + (w_x, v_x)(u_x,w_{xx}) + (v_x, v_x)(w_x,w_{xx})\\
     &~ + \frac 32 \int_{S^1}(w, Au)(u_x,w_{xx}) + (w, Av)(u_x,w_{xx})+(v, Av)(w_x,w_{xx})\\
    \le&~ C(\norm{u}_{W^{3,2}} +\norm{v}_{W^{3,2}})\norm{w}_{W^{1,2}}^2.
\end{split}
\end{equation}
We shall examine this term by term carefully. First of all, it's
obvious
\[ \int_{S^1}(w_{xxx}, w_{xx}) =0. \]
Also, it's easy to check
\begin{equation*}
\begin{split}
3\int_{S^1}(v_x, v_{xx})(w,w_{xx}) + &\frac 32 \int_{S^1}(w, Au)(u_x,w_{xx}) + (w, Av)(u_x,w_{xx})\\
\le&~ C(\norm{u}_{W^{3,2}} +\norm{v}_{W^{3,2}})\norm{w}_{W^{1,2}}^2.
\end{split}
\end{equation*}
Furthermore, by integrating by parts, we have
\begin{equation*}
\begin{split}
&\frac 32 \int_{S^1}(w_x, u_x)(u_x,w_{xx}) + (v_x, v_x)(w_x,w_{xx}) + (v, Av)(w_x,w_{xx})\\
=&~-\frac 32 \int_{S^1}(w_x, u_x)(u_{xx},w_x) + (v_x, v_{xx})(w_x,w_x) + (v, Av_x)(w_x,w_x)\\
\le&~ C(\norm{u}_{W^{3,2}} +\norm{v}_{W^{3,2}})\norm{w_x}_{L^2}^2.
\end{split}
\end{equation*}
Similarly,
\begin{equation*}
\begin{split}
&~\frac 32 \int_{S^1}(w_x, v_x)(u_x,w_{xx})\\
=&~ \frac 32 \int_{S^1}(w_x, v_x)(w_x,w_{xx}) + (w_x, v_x)(v_x,w_{xx})\\
\le&~\frac 32
\norm{v_x}_{L^\infty}\norm{w_{xx}}_{L^\infty}\norm{w_x}_{L^2}^2
        -\frac 32 \int_{S^1}(w_x, v_x)(v_{xx},w_x)\\
\le&~ C(\norm{u}_{W^{3,2}} +\norm{v}_{W^{3,2}})\norm{w_x}_{L^2}^2.
\end{split}
\end{equation*}
Thus there are only two terms left, i.e.
\begin{equation}\label{u7}
3\int_{S^1}(u_x, w_{xx})(u, w_{xx}) \quad\mbox{and}\quad
3\int_{S^1}(w_x, v_{xx})(u, w_{xx}).
\end{equation}
To treat them, we observe that $|u|^2 =1$ implies $(u,u_x)=0$, hence
$(u,u_{xx}) + (u_x, u_x)=0$. Therefore,
\begin{equation*}
\begin{split}
(u,w_{xx}) =& (u, u_{xx} - v_{xx}) = -(u_x, u_x) - (u, v_{xx})\\
=&~-(u_x, u_x) + (u_x, v_x) - (u_x, v_x) + (v_x, v_x) + (v,v_{xx}) - (u,v_{xx})\\
=&~-(u_x, w_x) - (w_x, v_x) - (w,v_{xx}).
\end{split}
\end{equation*}
Taking this into account, we can bound~(\ref{u7}) in the same way as
above. Namely, we have
\[ 3\int_{S^1}(u_x, w_{xx})(u, w_{xx}) +3\int_{S^1}(w_x, v_{xx})(u, w_{xx})
\le~ C(\norm{u}_{W^{3,2}}
+\norm{v}_{W^{3,2}})\norm{w}_{W^{1,2}}^2.\] So we proved the claim
and finally get from (\ref{u5}) and (\ref{u6}) that
\[ \dt \norm{w}_{W^{1,2}}^2 \le C(\norm{u}_{W^{3,2}} +\norm{v}_{W^{3,2}})\norm{w}_{W^{1,2}}^2. \]
Thus Lemma~\ref{l2} implies
\[ \norm{w(t)}_{W^{1,2}}^2\le C\norm{w(0)}_{W^{1,2}}^2. \]
Since $u$ and $v$ share the same initial data, we know $w(0) = 0$.
Hence we conclude that $w(t) = 0$, i.e. the solution is unique.
\end{proof}

\section{Semi-conservation laws}
\label{Sec:4}

After getting a local solution $u\in
L^\infty([0,T),W^{k,2}(S^1,S^n))$ of Cauchy problem
~(\ref{e:cauchy}), what we need to do next is to derive
 an uniform estimate of $\norm{u}_{W^{3,2}}$ for all $t\in [0, T)$. Then Theorem~\ref{t1}
 ensures that the local solution can be extended to a global solution. In geometric evolution
problems, this is usually done by finding some energy conservation
laws, see~\cite{DW1,SW,WW} for example. However, in the current
situation we fail to find conservation quantities except the energy
$E_1(u)$. Nevertheless, we do find semi-conservation laws for two
higher order energies which is sufficient to prove the global
existence.

\begin{rem}
Our goal here is to bound the norm $\norm{u_x}_{H^{2,2}}$, which is
equivalent to the Sobolev norm $\norm{u}_{W^{3,2}}$, for all
existing time $t\in [0,T)$. However, there are some unexpected terms
emerging, when we calculate $\frac{d}{dt}\norm{u_x}_{H^{2,2}}$
directly. These `bad' terms can't be controlled by the linear form
of $\norm{u_x}_{H^{2,2}}$, which is necessary when carrying out
Gronwall's inequality. Luckily, we find some other energies which
satisfy semi-conservative laws. More importantly, these
semi-conservative laws implies a global bound of
$\norm{u}_{W^{3,2}}$.
\end{rem}

Through out this section, we let $u\in
L^\infty([0,T),W^{k,2}(S^1,S^n))$ be a local solution of Cauchy
problem ~(\ref{e:cauchy}) on the time interval $[0,T)$, and use $C$
to denote constants which may depend on the initial data, the matrix
$A$ and the maximal time $T$. First we define a second order
`energy'
\begin{equation*}
    E_{2}(u)
    =\int_{S^{1}}|\nabla_{x}u_{x}|^{2}
    -\frac{1}{4}\int_{S^{1}}|u_{x}|^4
    -\frac{9}{4}\int_{S^{1}}(u,Au)|u_{x}|^2
    +\int_{S^{1}}(u_{x},Au_{x}).
\end{equation*}
To derive the semi-conservation law, we take the time derivative
 \begin{equation}\label{eq:4.1}
    \frac d{dt}E_{2}(u)
    =\frac d{dt}\int_{S^{1}}|\nabla_{x}u_{x}|^{2}
    -\frac{1}{4}\frac d{dt}\int_{S^{1}}|u_{x}|^4
    -\frac{9}{4}\frac d{dt}\int_{S^{1}}(u,Au)|u_{x}|^2
    +\frac d{dt}\int_{S^{1}}(u_{x},Au_{x})
\end{equation}
and compute the four terms in~(\ref{eq:4.1}) one by one.

Using the equation~(\ref{e2}) and changing the order of derivatives,
we have
\begin{equation*}
    \begin{split}
        &\frac{d}{dt}\int_{S^{1}}|\nabla_{x}u_{x}|^2\\
        = &
        2\int_{S^{1}}\langle\nabla_{t}\nabla_{x}u_{x},\nabla_{x}u_{x}\rangle\\
        =&
        2\int_{S^{1}}\langle\nabla_{x}\nabla_{x}u_{t},\nabla_{x}u_{x}\rangle
        +2\int_{S^{1}}\langle R(u_{x},u_{t})u_{x},\nabla_{x}u_{x}\rangle\\
        = &
        2\int_{S^{1}}\langle u_{t},\nabla_{x}^{3}u_{x}\rangle
        +2\int_{S^{1}}\langle
        R(u_{x},\nabla_{x}^{2}u_{x})u_{x},\nabla_{x}u_{x}\rangle\\
          &
        +\int_{S^{1}}\langle
        R(u_{x},u_{x})u_{x},\nabla_{x}u_{x}\rangle|u_{x}|^2
        +3\int_{S^{1}}\langle R(u_{x},u_{x})u_{x},\nabla_{x}u_{x}\rangle(u,Au)\\
        = &
        2\int_{S^{1}}\langle u_{t},\nabla_{x}^{3}u_{x}\rangle
        +2\int_{S^{1}}\langle
        R(u_{x},\nabla_{x}^{2}u_{x})u_{x},\nabla_{x}u_{x}\rangle\\
        := &
        I_{1}+I_{2}\\
    \end{split}
\end{equation*}
Here we noticed that the curvature tensor $R$ on $S^{n}$ is
constant, hence $\nx R \equiv 0$. We first compute the term $I_1$.
\begin{equation*}
    \begin{split}
        I_{1}
        = &
        2\int_{S^{1}}\langle
        \nabla_{x}^{2}u_{x}
        +\frac{1}{2}\langle u_{x},u_{x}\rangle u_{x}
        +\frac{3}{2}(u,Au)
        u_{x},\nabla_{x}^{3}u_{x}\rangle\\
        = &
        2\int_{S^{1}}\langle
        \nabla_{x}^{2}u_{x},\nabla_{x}^{3}u_{x}\rangle
        +\int_{S^{1}}\langle u_{x},u_{x}\rangle\langle u_{x},\nabla_{x}^{3}u_{x}\rangle
        +3\int_{S^{1}}(u,Au)
        \langle u_{x},\nabla_{x}^{3}u_{x}\rangle\\
        = &
        -2\int_{S^{1}}\langle
        \nabla_{x}u_{x},u_{x}\rangle\langle u_{x},\nabla_{x}^{2}u_{x}\rangle
        -\int_{S^{1}}\langle u_{x},u_{x}\rangle\langle\nabla_{x}
        u_{x},\nabla_{x}^{2}u_{x}\rangle\\
          &
        -6\int_{S^{1}}(u_{x},Au)\langle
        u_{x},\nabla_{x}^{2}u_{x}\rangle
        -3\int_{S^{1}}(u,Au)\langle\nabla_{x}
        u_{x},\nabla_{x}^{2}u_{x}\rangle\\
        = &
        2\int_{S^{1}}\langle
        \nabla_{x}u_{x},u_{x}\rangle\langle\nabla_{x} u_{x},\nabla_{x}u_{x}\rangle
        +\int_{S^{1}}\langle\nabla_{x} u_{x},u_{x}\rangle\langle\nabla_{x}
        u_{x},\nabla_{x}u_{x}\rangle\\
          &
        +6\int_{S^{1}}(u_{xx},Au)\langle
        u_{x},\nabla_{x}u_{x}\rangle
        +6\int_{S^{1}}(u_{x},Au_{x})\langle
        u_{x},\nabla_{x}u_{x}\rangle\\
          &
        +6\int_{S^{1}}(u_{x},Au)\langle
        \nabla_{x}u_{x},\nabla_{x}u_{x}\rangle
        +3\int_{S^{1}}(u_{x},Au)\langle\nabla_{x}
        u_{x},\nabla_{x}u_{x}\rangle\\
        = &
        3\int_{S^{1}}\langle
        \nabla_{x}u_{x},u_{x}\rangle\langle\nabla_{x} u_{x},\nabla_{x}u_{x}\rangle
        +6\int_{S^{1}}(u_{xx},Au)\langle
        u_{x},\nabla_{x}u_{x}\rangle\\
          &
        +6\int_{S^{1}}(u_{x},Au_{x})\langle
        u_{x},\nabla_{x}u_{x}\rangle
        +9\int_{S^{1}}(u_{x},Au)\langle\nabla_{x}
        u_{x},\nabla_{x}u_{x}\rangle
    \end{split}
\end{equation*}
For $I_{2}$, since $\nx R \equiv 0$, we have
\begin{equation*}
    \begin{split}
        I_{2}
        = &
        2\int_{S^{1}}\langle
        R(u_{x},\nabla_{x}^{2}u_{x})u_{x},\nabla_{x}u_{x}\rangle\\
        = &
        \int_{S^{1}}\langle
        R(u_{x},\nabla_{x}^{2}u_{x})u_{x},\nabla_{x}u_{x}\rangle
        -\int_{S^{1}}\langle
        R(\nabla_{x}u_{x},\nabla_{x}u_{x})u_{x},\nabla_{x}u_{x}\rangle\\
          &
        -\int_{S^{1}}\langle
        (\nabla_{x}R)(u_{x},\nabla_{x}u_{x})u_{x},\nabla_{x}u_{x}\rangle
        -\int_{S^{1}}\langle
        R(u_{x},\nabla_{x}u_{x})\nabla_{x}u_{x},\nabla_{x}u_{x}\rangle\\
          &
        -\int_{S^{1}}\langle
        R(u_{x},\nabla_{x}u_{x})u_{x},\nabla_{x}^{2}u_{x}\rangle\\
         = &0.
    \end{split}
\end{equation*}
Thus for the first term in~(\ref{eq:4.1}), we get
\begin{equation}\label{eq:4.2}
    \begin{split}
        &\frac{d}{dt}\int_{S^{1}}|\nabla_{x}u_{x}|^2\\
        = &
        3\int_{S^{1}}\langle
        \nabla_{x}u_{x},u_{x}\rangle\langle\nabla_{x} u_{x},\nabla_{x}u_{x}\rangle
        +6\int_{S^{1}}(u_{xx},Au)\langle
        u_{x},\nabla_{x}u_{x}\rangle\\
          &
        +6\int_{S^{1}}(u_{x},Au_{x})\langle
        u_{x},\nabla_{x}u_{x}\rangle
        +9\int_{S^{1}}(u_{x},Au)\langle\nabla_{x}
        u_{x},\nabla_{x}u_{x}\rangle.
    \end{split}
\end{equation}
On the other hand,
\begin{equation}\label{eq:4.3}
    \begin{split}
        &\frac{d}{dt}\int_{S^{1}}|u_{x}|^4\\
        = &
        4\int_{S^{1}}\langle\nabla_{t} u_{x}, u_{x}\rangle\langle u_{x}, u_{x}\rangle\\
        = &
        -4\int_{S^{1}}\langle u_{t},\nabla_{x}u_{x}\rangle\langle u_{x}, u_{x}\rangle
        -8\int_{S^{1}}\langle u_{t},u_{x}\rangle\langle u_{x},
        \nabla_{x}u_{x}\rangle\\
        = &
        -4\int_{S^{1}}\langle\nabla_{x}^{2} u_{x}, \nabla_{x}u_{x}\rangle\langle u_{x}, u_{x}\rangle
        -2\int_{S^{1}}\langle u_{x}, u_{x}\rangle\langle u_{x}, u_{x}\rangle\langle u_{x},
        \nabla_{x}u_{x}\rangle\\
          &
        -6\int_{S^{1}}(u,Au)\langle u_{x},\nabla_{x}u_{x}\rangle\langle u_{x}, u_{x}\rangle
        -8\int_{S^{1}}\langle\nabla_{x}^{2} u_{x}, u_{x}\rangle
        \langle u_{x},\nabla_{x}u_{x}\rangle\\
          &
        -4\int_{S^{1}}\langle u_{x}, u_{x}\rangle\langle u_{x}, u_{x}\rangle
        \langle u_{x},\nabla_{x}u_{x}\rangle
        -12\int_{S^{1}}(u,Au)\langle u_{x}, u_{x}\rangle
        \langle u_{x},\nabla_{x}u_{x}\rangle\\
        = &
        4\int_{S^{1}}\langle\nabla_{x} u_{x}, \nabla_{x}u_{x}\rangle\langle u_{x}, \nabla_{x}u_{x}\rangle
        -6\int_{S^{1}}(u,Au)\langle u_{x},\nabla_{x}u_{x}\rangle\langle u_{x}, u_{x}\rangle\\
          &
        +8\int_{S^{1}}\langle\nabla_{x} u_{x}, u_{x}\rangle
        \langle\nabla_{x} u_{x},\nabla_{x}u_{x}\rangle
        -12\int_{S^{1}}(u,Au)\langle u_{x}, u_{x}\rangle
        \langle u_{x},\nabla_{x}u_{x}\rangle\\
        = &
        12\int_{S^{1}}\langle\nabla_{x} u_{x}, \nabla_{x}u_{x}\rangle\langle u_{x}, \nabla_{x}u_{x}\rangle
        +9\int_{S^{1}}(u_{x},Au)\langle u_{x},u_{x}\rangle\langle u_{x},
        u_{x}\rangle.
    \end{split}
\end{equation}
Besides,
\begin{equation}\label{eq:4.4}
    \begin{split}
        \lefteqn{\frac{d}{dt}\int_{S^{1}}(u,Au)\langle u_{x},
        u_{x}\rangle}\\
        = &
        2\int_{S^{1}}(u_{t},Au)\langle u_{x}, u_{x}\rangle
        +2\int_{S^{1}}(u,Au)\langle\nabla_{t} u_{x},
        u_{x}\rangle\\
        = &
        2\int_{S^{1}}(u_{t},Au)\langle u_{x}, u_{x}\rangle
        -2\int_{S^{1}}(u,Au)\langle u_{t},
        \nabla_{x}u_{x}\rangle
        -4\int_{S^{1}}(u_{x},Au)\langle u_{t},
        u_{x}\rangle\\
        = &
        2\int_{S^{1}}(\nabla_{x}^{2}u_{x}, Au)\langle u_{x},
        u_{x}\rangle
        +\int_{S^{1}}(u_{x},Au)\langle u_{x},
        u_{x}\rangle^2\\
          &
        +3\int_{S^{1}}(u,Au)(u_{x},Au)\langle u_{x},
        u_{x}\rangle
        -2\int_{S^{1}}(u,Au)\langle \nabla_{x}^{2}u_{x},
        \nabla_{x}u_{x}\rangle\\
          &
        -\int_{S^{1}}(u,Au)\langle u_{x}, u_{x}\rangle\langle u_{x},
        \nabla_{x}u_{x}\rangle
        -3\int_{S^{1}}(u,Au)^2\langle u_{x},
        \nabla_{x}u_{x}\rangle\\
          &
        -4\int_{S^{1}}(u_{x},Au)\langle \nabla_{x}^{2}u_{x},
        u_{x}\rangle
        -2\int_{S^{1}}(u_{x},Au)\langle u_{x},
        u_{x}\rangle^2\\
          &
        -6\int_{S^{1}}(u_{x},Au)(u,Au)\langle u_{x},
        u_{x}\rangle\\
        = &
        2\int_{S^{1}}(\nabla_{x}^{2}u_{x}, Au)\langle u_{x},
        u_{x}\rangle
        -\frac 12\int_{S^{1}}(u_{x},Au)\langle u_{x},
        u_{x}\rangle^2\\
          &
        +3\int_{S^{1}}(u,Au)(u_{x},Au)\langle u_{x},
        u_{x}\rangle
        +4\int_{S^{1}}(u_x,Au)\langle \nabla_{x}u_{x},
        \nabla_{x}u_{x}\rangle\\
          &
        +4\int_{S^{1}}(u_{xx},Au)\langle \nabla_{x}u_{x},
        u_{x}\rangle
        +4\int_{S^{1}}(u_{x},Au_x)\langle \nx u_{x},
        u_{x}\rangle\\
    \end{split}
\end{equation}
To proceed, we recall that
$$\nabla_{x}^{2}u_{x}=u_{xxx}+3(u_{x},u_{xx}) u+\langle
u_{x},u_{x}\rangle u_{x}.$$ So for the first term in the last
equality of (\ref{eq:4.4}), we have
\begin{equation}\label{eq:4.5}
    \begin{split}
        &2\int_{S^{1}}(\nabla_{x}^{2}u_{x}, Au)\langle u_{x},
        u_{x}\rangle\\
        =&
        2\int_{S^{1}}(u_{xxx}, Au)\lg u_x, u_x \rg+6\int_{S^{1}}\lg u_x, u_{xx} \rg(u, Au)\lg u_x, u_x \rg \\
        &  +2\int_{S^{1}}\lg u_x, u_x \rg^2(u_x, Ax)\\
        =&
        2\int_{S^{1}}(u_{x},Au_x)\langle \nx u_{x},
        u_{x}\rangle-4\int_{S^{1}}(u_{xx},Au)\langle \nabla_{x}u_{x},
        u_{x}\rangle\\
        &-\int_{S^{1}}(u_{x},Au)\langle u_{x},
        u_{x}\rangle^2
    \end{split}
\end{equation}
(\ref{eq:4.4}) and (\ref{eq:4.5}) yields
\begin{equation}\label{eq:4.6}
    \begin{split}
        \lefteqn{\frac{d}{dt}\int_{S^{1}}(u,Au)\langle u_{x},
        u_{x}\rangle}\\
        =&
        4\int_{S^{1}}(u_x,Au)\langle \nabla_{x}u_{x},
        \nabla_{x}u_{x}\rangle+6\int_{S^{1}}(u_{x},Au_x)\langle \nx u_{x},
        u_{x}\rangle\\
        &
        +3\int_{S^{1}}(u,Au)(u_{x},Au)\langle u_{x},
        u_{x}\rangle-\frac 32\int_{S^{1}}(u_{x},Au)\langle u_{x},
        u_{x}\rangle^2
        \\
    \end{split}
\end{equation}
For the last term in~(\ref{eq:4.1}), we have
\begin{equation}\label{eq:4.7}
    \begin{split}
        \lefteqn{\frac{d}{dt}\int_{S^{1}}(u_{x},
        Au_{x})}\\
        = &
        2\int_{S^{1}}(u_{xt}, Au_{x})
        =-2\int_{S^{1}}(u_{t}, Au_{xx})\\
        = &
        -2\int_{S^{1}}\langle \nabla_{x}^{2}u_{x}
        +\frac{1}{2}\langle u_{x}, u_{x}\rangle u_{x}
        +\frac{3}{2}(u,Au) u_{x}, Au_{xx}\rangle\\
        = &
        -2\int_{S^{1}}(\nabla_{x}^{2}u_{x}, Au_{xx})
        -\int_{S^{1}}\langle u_{x}, u_{x}\rangle(u_{x},Au_{xx})
        -3\int_{S^{1}}(u,Au)(u_{x},Au_{xx})\\
        = &
        -2\int_{S^{1}}(u_{xxx}, Au_{xx})
        -6\int_{S^{1}}(u, Au_{xx})\langle u_{x}, \nabla_{x}u_{x}\rangle
        -2\int_{S^{1}}(u_{x},Au_{xx})\langle u_{x}, u_{x}\rangle\\
          &
        -\int_{S^{1}}\langle u_{x}, u_{x}\rangle(u_{x},
        Au_{xx})
        +3\int_{S^{1}}(u_{x},Au)(u_{x},Au_{x})\\
        = &
        -6\int_{S^{1}}(u, Au_{xx})\langle u_{x}, \nabla_{x}u_{x}\rangle
        +3\int_{S^{1}}(u_{x},Au_{x})\langle u_{x},
        \nabla_{x}u_{x}\rangle\\
          &
        +3\int_{S^{1}}(u_{x},Au)(u_{x},Au_{x})
    \end{split}
\end{equation}
Here we notice that $\nabla_{x}u_{x}=u_{xx}+\langle
u_{x},u_{x}\rangle u$ and $\nabla_{x}^{2}u_{x}=u_{xxx}+3(u_{x},u_{xx})u+\langle u_{x},u_{x}\rangle u_{x}$.\\
Combining (\ref{eq:4.2}),(\ref{eq:4.3}),(\ref{eq:4.6})
and(\ref{eq:4.7}), we finally get
\begin{equation}\label{eq:4.8}
    \begin{split}
        \frac{d}{dt}E_2(u)=&
        \frac{9}{4}\int_{S^{1}}(u_{x},Au)|u_{x}|^{4}
        +3\int_{S^{1}}(u_{x},Au)(u_{x},
        Au_{x})\\
        &-\frac 92\int_{S^{1}}(u_{x},Au_x)\langle \nx u_{x}, u_x \rangle
        -\frac{27}{4}\int_{S^{1}}(u,Au)(u_{x},Au)\langle u_{x},
        u_{x}\rangle.
    \end{split}
\end{equation}
Now we can derive the desired Gronwall-type inequality. Since $A$ is
a constant matrix and $|u|=1$, (\ref{eq:4.8}) yields
\begin{equation}\label{eq:4.9}
    \begin{split}
        \frac{d}{dt}E_2(u)&\le C\int_{S^{1}}|u_{x}|^{5}+C\int_{S^{1}}|u_{x}|^{3}+ C\int_{S^{1}}|u_{x}|^{3}|\nx u_x|\\
        &\le C\int_{S^{1}}|u_{x}|^{5}+C\int_{S^{1}}|u_{x}|^{3}|\nx u_x|.
            \end{split}
\end{equation}

At this point, we may recall that Lemma~\ref{l3} provides the
desired bounds for both $\norm{u_x}_{L^2}$ and
$\norm{u_x}_{L^\infty}$. Therefore,
\begin{equation}\label{eq:4.10}
    \begin{split}
        \int_{S^{1}}|u_{x}|^{5}\le& \norm{u_x}_{L^\infty}^3\int_{S^{1}}|u_{x}|^{2}\\
        \le& C\norm{u_x}_{H^{1,2}}^{3/2}\norm{u_x}_{L^2}^2\\
        \le& C\int_{S^{1}}|\nabla_{x}u_{x}|^{2}+C.
    \end{split}
\end{equation}
Furthermore,
\begin{equation}\label{eq:4.11}
    \begin{split}
        \int_{S^{1}}|u_{x}|^{3}|\nx u_x|\le& \norm{u_x}_{L^\infty}^2\int_{S^{1}}|u_{x}||\nx u_x|\\
        \le& C\norm{u_x}_{H^{1,2}}\cdot\norm{u_x}_{L^2}\norm{\nx u_x}_{L^2} \\
        \le& C\int_{S^{1}}|\nabla_{x}u_{x}|^{2}+C.
    \end{split}
\end{equation}
So we arrive at
\begin{equation}\label{ee7}
    \frac{d}{dt}E_2(u)\le C\int_{S^{1}}|\nabla_{x}u_{x}|^{2}+C.
\end{equation}

Next, we claim that the integral $\int_{S^{1}}|\nabla_{x}u_{x}|^{2}$
is controlled by $E_2(u)$. Indeed,
\begin{equation}\label{ee2}
    \begin{split}
        \int_{S^{1}}|\nabla_{x}u_{x}|^{2}
        =&
        E_{2}(u)
        +\frac{1}{4}\int_{S^{1}}|u_{x}|^{4}
        +\frac{9}{4}\int_{S^{1}}(u,Au)|u_{x}|^{2}
        -\int_{S^{1}}(u_{x},Au_{x})\\
        \leq&
        E_2(u) + C\int_{S^{1}}|u_{x}|^{4}+C.
    \end{split}
\end{equation}
By the same virtual of the estimate~(\ref{eq:4.10}), we have
\begin{equation}\label{ee3}
    \begin{split}
        \int_{S^{1}}|u_{x}|^{4}\le& \norm{u_x}_{L^\infty}^2\int_{S^{1}}|u_{x}|^{2}\\
        \le& C\norm{u_x}_{H^{1,2}}\norm{u_x}_{L^2}^2\\
        \le& C(\epsilon\norm{u_x}_{H^{1,2}}^{2}+\frac 1\epsilon\norm{u_x}_{L^2}^4)\\
        \le& C\epsilon\int_{S^{1}}|\nabla_{x}u_{x}|^{2} + C.
    \end{split}
\end{equation}
Here we employed Young's inequality with $\epsilon$,i.e.
\[ ab\le \frac {\epsilon a^2}{2} + \frac {b^2}{2\epsilon},\text{~~for~} a,b>0. \]
Thus if we choose $\epsilon$ sufficiently small, we proved the claim
from (\ref{ee2}) and (\ref{ee3}) that
\begin{equation}\label{ee4}
\int_{S^{1}}|\nabla_{x}u_{x}|^{2}\le CE_2(u) +C.
\end{equation}
Consequently, we conclude from (\ref{ee7}) and(\ref{ee4}) that
\begin{equation}\label{ee5}
\frac d{dt}E_2(u) \le CE_2(u) + C.
\end{equation}
By Gronwall's inequality, we finally arrive at
\begin{lem}\label{l4}
Suppose $u:S^1\times[0,T)\to S^n$ is a solution to the cauchy
problem~(\ref{e:cauchy}), then  for all $t\in [0,T)$
$$E_2(u(t)) \le C(T) , \int_{S^{1}}|\nabla_{x}u_{x}(t)|^{2}dx\le C(T),$$
where $C(T)$ is a constant depending on $T$ and the initial data
$u_0$.
\end{lem}

Our last ingredient in proving the global existence is the
semi-conservation law for a third order energy, which is given by
\begin{equation}\label{eq:4.12}
    E_{3}(u)
    =\int_{S^{1}}|\nabla_{x}^{2} u_{x}|^{2}
    -\int_{S^{1}}\langle u_{x},\nabla_{x}u_{x}\rangle^{2}
    -\frac{3}{2}\int_{S^{1}}|u_{x}|^{2}|\nabla_{x}u_{x}|^{2}.
\end{equation}
Similarly,  we have the following lemma.
\begin{lem}\label{l5}
Suppose $u:S^1\times[0,T)\to S^n$ is a solution to the cauchy
problem~(\ref{e:cauchy}), then  for all $t\in [0,T)$
$$E_3(u(t)) \le C(T) , \int_{S^{1}}|\nabla_{x}^2u_{x}(t)|^{2}dx\le C(T),$$
where $C(t)$ is a constant depending on $T$ and the initial data
$u_0$.
\end{lem}
It takes a lot of efforts to find the energy functional $E_3$ which
satisfies the semi-conservation law and therefore bounded under the
flow. The spirit is all the same as that of $E_{2}$ demonstrated in
the proof of Lemma~\ref{l4}. So we are going to omit the lengthy
computation which mainly involves integration by parts and changing
orders of derivatives, and only give the key steps instead.

The time-derivative of $E_3(u)$ consists of four terms. A long
computation yields
\begin{eqnarray}\label{eq:4.17}
    \begin{split}
        &\frac{d}{dt}\int_{S^{1}}|\nabla_{x}^{2}
        u_{x}|^2\\
        = &
        2\int_{S^{1}}\langle\nabla_{x}^{3}u_{t},\nabla_{x}^{2}u_{x}\rangle
        +2\int_{S^{1}}\langle
        \nabla_{x}(R(u_{t},u_{x})u_{x}),\nabla_{x}^{2}u_{x}\rangle\\
          &
        +2\int_{S^{1}}\langle R(u_{t},u_{x})\nabla_{x}u_{x},\nabla_{x}^{2}u_{x}\rangle\\
        = &
        6\int_{S^{1}}\langle u_{x},\nabla_{x}^{2}u_{x}\rangle
        \langle\nabla_{x} u_{x}, \nabla_{x}^{2}u_{x}\rangle
        +9\int_{S^{1}}\langle u_{x}, \nabla_{x}u_{x}\rangle|\nabla_{x}^{2}
        u_{x}|^2\\
          &
        +6\int_{S^{1}}(u_{xxx},Au)\langle
        u_{x},\nabla_{x}^{2}u_{x}\rangle
        +15\int_{S^{1}}(u_{x},Au)|\nabla_{x}^{2}u_{x}|^2\\
          &
        +18\int_{S^{1}}(u_{xx},Au)\langle
        \nabla_{x}u_{x},\nabla_{x}^{2}u_{x}\rangle
        +18\int_{S^{1}}(u_{x},Au_{x})\langle
        \nabla_{x}u_{x},\nabla_{x}^{2}u_{x}\rangle\\
          &
        +18\int_{S^{1}}(u_{xx},Au_{x})\langle
        u_{x},\nabla_{x}^{2}u_{x}\rangle.
    \end{split}
\end{eqnarray}
For the second term, we have
\begin{equation}\label{eq:4.18}
    \begin{split}
        &\frac{d}{dt}\int_{S^{1}}\langle
        u_{x},\nabla_{x}u_{x}\rangle^2\\
        = &
        2\int_{S^{1}}\langle\nabla_{x} u_{t},\nabla_{x}u_{x}\rangle\langle
        u_{x},\nabla_{x}u_{x}\rangle
        +2\int_{S^{1}}\langle u_{x},\nabla_{x}\nabla_{t}u_{x}\rangle\langle
        u_{x},\nabla_{x}u_{x}\rangle\\   &
        +2\int_{S^{1}}\langle u_{x},R(u_{t},u_{x})u_{x}\rangle\langle
        u_{x},\nabla_{x}u_{x}\rangle\\
        = &
        6\int_{S^{1}}\langle u_{x},\nabla_{x}^{2}u_{x}\rangle\langle
        \nabla_{x}u_{x},\nabla_{x}^{2}u_{x}\rangle
        +3\int_{S^{1}}\langle u_{x}, \nabla_{x}u_{x}\rangle^3\\
          &
        +15\int_{S^{1}}( u_{x}, Au)\langle u_{x},
        \nabla_{x}u_{x}\rangle^2
        +6\int_{S^{1}}| u_{x}|^2( u_{xx}, Au)\langle
        u_{x},\nabla_{x}u_{x}\rangle\\
          &
        +6\int_{S^{1}}| u_{x}|^2( u_{x}, Au_{x})\langle
        u_{x},\nabla_{x}u_{x}\rangle\\
    \end{split}
\end{equation}
Besides, for the third term,
\begin{equation}\label{eq:4.19}
    \begin{split}
        &\frac{d}{dt}\int_{S^{1}}| u_{x}|^2| \nabla_{x}u_{x}|^2\\
        = &
        2\int_{S^{1}}\langle\nabla_{x} u_{t},u_{x}\rangle| \nabla_{x}u_{x}|^2
        +2\int_{S^{1}}| u_{x}|^2\langle
        \nabla_{x}\nabla_{x}u_{t},\nabla_{x}u_{x}\rangle\\
          &
        +2\int_{S^{1}}| u_{x}|^2\langle
        R(u_{t},u_{x})u_{x},\nabla_{x}u_{x}\rangle\\
         = &
        6\int_{S^{1}}\langle u_{x},\nabla_{x}u_{x}\rangle|\nabla_{x}^{2}u_{x}|^2
        +15\int_{S^{1}}( u_{x}, Au)| u_{x}|^2| \nabla_{x}u_{x}|^2\\
          &
        +3\int_{S^{1}}| u_{x}|^2\langle
        u_{x},\nabla_{x}u_{x}\rangle| \nabla_{x}u_{x}|^2
        +6\int_{S^{1}}| u_{x}|^2( u_{xx}, Au)\langle
        u_{x},\nabla_{x}u_{x}\rangle\\
          &
        +6\int_{S^{1}}| u_{x}|^2( u_{x}, Au_{x})\langle
        u_{x},\nabla_{x}u_{x}\rangle\\
    \end{split}
\end{equation}
Putting these three terms together, we get
\begin{equation}\label{eq:4.21}
    \begin{split}
        \frac{d}{dt}E_3(u)        = &
        6\int_{S^{1}}(u_{xxx},Au)\langle
        u_{x},\nabla_{x}^{2}u_{x}\rangle
        +15\int_{S^{1}}(u_{x},Au)\langle
        \nabla_{x}^{2}u_{x},\nabla_{x}^{2}u_{x}\rangle\\
          &
        -3\int_{S^{1}}\langle u_{x}, \nabla_{x}u_{x}\rangle^3 + \{\text{lower order terms}\}.\\
    \end{split}
\end{equation}
The key point in the above is that all the `bad' terms, i.e. the
higher order terms including $$\int_{S^{1}}\langle
        u_{x},\nabla_{x}^{2}u_{x}\rangle\langle\nabla_{x} u_{x},
        \nabla_{x}^{2}u_{x}\rangle
        ,\int_{S^{1}}\langle u_{x}, \nabla_{x}u_{x}\rangle|\nabla_{x}^{2}
        u_{x}|^2$$ vanish in the
summation. Because we have already shown that $\norm{u_x}_{H^{1,2}}$
and $\norm{u_x}_{L^\infty}$ are bounded by Lemma~\ref{l3} and
Lemma~\ref{l4}. The other terms left, though seems a lot, are all
controllable. Here we only take the term $\int_{S^{1}}\langle u_{x},
\nabla_{x}u_{x}\rangle^3$ for example to demonstrate this. Similar
to the estimate of $E_2$, the interpolation inequality--namely,
Corollary~\ref{c1}--plays an important role here.
\begin{equation}\label{eq:4.22}
    \begin{split}
        \int_{S^{1}}\langle u_{x}, \nabla_{x}u_{x}\rangle^3 &\le
        C\norm{u_x}_{L^\infty}^3\norm{\nx u_x}_{L^\infty}^3\\
        &\leq
        C\norm{u_x}_{H^{1,2}}^{\frac 32}\cdot\norm{\nx u_x}_{H^{1,2}}^{\frac
        32}\norm{\nx u_x}_{L^2}^{\frac 32}\\
        &\leq
        C\int_{S^{1}}|\nabla_{x}^{2}u_{x}|^{2}+C.
    \end{split}
\end{equation}
Finally, we can get
\begin{equation}\label{eq:4.23}
    \frac{d}{dt}E_3(u) \le C\int_{S^{1}}|\nabla_{x}^{2}u_{x}|^{2}+C \le CE_3(u) + C.
\end{equation}
Here, the constant $C$ depends on $T$ and $\norm{u_0}_{W^{2,2}}$.
Thus, $E_{3}$ is bounded in $[0,T)$ by Gronwall's inequality.
Moreover, we also have
\begin{equation}\label{eq:4.24}
    \begin{split}
        \int_{S^{1}}|\nabla_{x}^{2}u_{x}|^{2}\leq C(T).
    \end{split}
\end{equation}
\\

\section{Global existence}
\label{Sec:5}

In this section we finish the proof of Theorem \ref{t1}.

Let $u$ be the local smooth solution of (\ref{e:cauchy}) which
exists on the maximal time interval $[0,T)$. If $T=\infty$, then
Theorem \ref{t1} holds true. Thus we only need to consider the case
where $T<\infty$.

From Lemma~\ref{l4} and Lemma~\ref{l5}, we have
\begin{equation*}
  \sup_{t\in[0,T)}\norm{u_x}_{H^{2,2}} \le C,
\end{equation*}
where $C$ depends on $T$ and the $W^{3,3}$-norm of the initial data
$u_0$. By Theorem~\ref{t:equivalence}, the $W^{3,2}$-Sobolev norm of
$u$ is bounded by
\begin{equation}\label{eee6}
  \sup_{t\in[0,T)}\norm{u}_{W^{3,2}} \le C\sup_{t\in[0,T)}\norm{u_x}_{H^{2,2}}^3
  \le C(T,\norm{u_0}_{W^{3,2}}).
\end{equation}

Thus by Theorem \ref{t3}, if $T$ is finite, we can find a local
solution $u_1$ of (\ref{e:cauchy}) satisfying the initial value
condition
$$u_1(x,0)=u(x,T-\epsilon),$$
where $0<\epsilon<T$ is a small number. By uniqueness
Theorem~\ref{t4}, we know that $u$ and $u_1$ coincides on the
overlapped time interval. Then from Lemma~\ref{l2}, one can see that
$u_1$ exists on the time interval $[0,T_1)$ with $T_1>0$ only
depending on the Sobolev norm $\norm{u(x,T-\epsilon)}_{W^{3,2}}$.
However, this norm is in turn decided by the initial data
$||u_0||_{W^{3,2}}$. The uniform bound~(\ref{eee6}) implies $T_1$ is
independent of $\epsilon$. Thus, by choosing $\epsilon$ sufficiently
small, we can glue $u$ and $u_1$ together to obtain a solution of
the Cauchy problem (\ref{e:cauchy}) on a larger time interval
$[0,T-\epsilon +T_1)$, where $T-\epsilon +T_1 >T$. This contradicts
to the maximality of $T$. Hence $T=\infty$ and the proof is done.

\section*{Acknowledgements}
The authors would like to thank Professor
Youde Wang for his inspiration and encouragement. They would also
like to thank Doctor Xiaowei Sun for many beneficial discussions.

%

\begin{thebibliography}{17}
%

\bibitem{A}
V. Adler; {Discretizations of the Landau-Lifshits equation}, Theor.
and Math. Phys. \textbf{124}(1), (2000)48¨C61.

\bibitem{BY}
M. Balakhnev; {Superposition formulas for integrable vector
evolution equations}, Theor. and Math. Phys. \textbf{154}(2),
(2008)261¨C267.

\bibitem{DC}
A. Degasperis, F. Calogero, \textit{Reduction technique for matrix
nonlinear evolution equations solvable by the spectral transform},
Preprint No. 151, Instituto di Fisica, Univ. di Roma, Rome (1979).

\bibitem{DW1}
W. Ding, Y. Wang; {Schr\"odinger flow of maps into symplectic
manifolds}, Sci. China Ser. A \textbf{41}(7) (1998), 746-755.

\bibitem{DW}
W. Ding, Y. Wang; {Local Schr\"odinger flow into {K\"ahler}
manifolds}, Sci. China Ser. A \textbf{44}(11) (2001), 1446-1464.

\bibitem{GS}
I. Golubchik, V. Sokolov: {Multicomponent generalization of the
hierarchy of the Landau- Lifshitz equation}. Theor. and Math. Phys.
\textbf{124}(1), 909-917 (2000).

\bibitem{ILMT}
S. Igonin, J. Vav De Leur, G. Manno, V. Trushkov; {Generalized
Landau-Lifshitz Systems and Lie Algebras Associated with Higher
Genus Curves}, arXiv:0811.4669v1 [nlin.SI] 28 Nov 2008

\bibitem{K}
C. Kenig, T. Lamm, D. Pollack, G. Staffilani, T. Toro; {The Cauchy
problem for Schr\"odinger flows into K\"ahler manifolds},
arXiv:0911.3141.

\bibitem{MS}
A. Meshkov, V. Sokolov; {Integrable Evolution Equations on the
N-Dimensional Sphere}. Comm. Math. Phys. \textbf{232},1-18(2002).

\bibitem{M}
R. Moser; {A Variational Problem Pertaining to Biharmonic Maps},
Communications in Partial Differential Equations, \textbf{33},
1654-1689, 2008

\bibitem{S}
T. Skrypnyk; {Deformations of loop algebras and integrable systems:
hierarchies of integrable equations}, Journal of Mathematical
Physics \textbf{45}(12), (2004).

\bibitem{SoW}
C. Song, Y. Wang; {Schr\"odinger Soliton from Lorentzian Manifolds},
arXiv:0910.1759.

\bibitem{SW}
X. Sun, Y. Wang; {KdV Geometric Flows on K\"ahler Manifolds},
preprint.

\bibitem{T}
M. Taylor; \textit{Partial Differential Equations III}, Springer
Verlag, New York, 1996.

\bibitem{V}
A.P. Veselov; {Finite-gap potentials and integrable systems on the
sphere with a quadradic potential}, Funct. Anal. and Appl.
\textbf{14}(1), 48-50(1980).

\bibitem{W}
Y. Wang; {Lecture on Geometric Flows on K\"ahler Manifolds},
preprint.

\bibitem{WW}
H. Wang, Y. Wang; {Global inhomogeneous Schrodinger flow},
International Journal Of Mathematics, \textbf{11}(8)
(2000),1079-1114.
\end{thebibliography}
%

\end{document}